\documentclass[12pt]{article}
\usepackage[all]{xy}

\usepackage{latexsym,amssymb,chicago,geometry,graphicx}
\usepackage{xspace}
\usepackage{amsthm}
\usepackage{amsfonts}
\usepackage{amsmath}
\usepackage{mathrsfs}
\mathcode`\:="603A     
\mathcode`\<="4268     
\mathcode`\>="5269     
\mathchardef\gt="313E  
\mathchardef\lt="313C  
\mathchardef\colon="303A  
\edef\cdrestoreat{
\noexpand\catcode\lq\noexpand\@=\the\catcode\lq\@}\catcode\lq\@=11
\def\thmitem{\def\theenumi{\@roman\c@enumi}
\def\labelenumi{{\normalfont(\theenumi)}}}
\expandafter\ifx\csname bezier\endcsname\relax
\newcounter{P@sc}\newcounter{P@scp}\newcounter{P@t}\newlength{\P@x}
\newlength{\P@xa}\newlength{\P@xb}\newlength{\P@y}\newlength{\P@ya}
\newlength{\P@yb}\newsavebox{\P@pt}
\def\bezier#1(#2,#3)(#4,#5)(#6,#7){\c@P@sc#1\relax
 \c@P@scp\c@P@sc \advance\c@P@scp\@ne
 \P@xb #4\unitlength \advance\P@xb -#2\unitlength \multiply\P@xb \tw@
 \P@xa #6\unitlength \advance\P@xa -#2\unitlength
 \advance\P@xa -\P@xb \divide\P@xa\c@P@sc
 \P@yb #5\unitlength \advance\P@yb -#3\unitlength \multiply\P@yb \tw@
 \P@ya #7\unitlength \advance\P@ya -#3\unitlength
 \advance\P@ya -\P@yb \divide\P@ya\c@P@sc
 \setbox\P@pt\hbox{\vrule height\@halfwidth depth\@halfwidth 
 width\@wholewidth}\c@P@t\z@ 
 \put(#2,#3){\@whilenum{\c@P@t<\c@P@scp}\do
 {\P@x\c@P@t\P@xa \advance\P@x\P@xb \divide\P@x\c@P@sc \multiply\P@x\c@P@t 
 \P@y\c@P@t\P@ya \advance\P@y\P@yb \divide\P@y\c@P@sc \multiply\P@y\c@P@t 
 \raise \P@y \hbox to \z@{\hskip \P@x\unhcopy\P@pt\hss}\advance\c@P@t\@ne}}}
\else\relax\fi
\def\sig{\mbox{\vbox to 0pt{\vss\hbox to 0pt{\setlength{\unitlength}{.35cm}
\begin{picture}(5,5)
      \bezier140(1,3)(1,0)(1.5,0) 
      \bezier120(1.5,0)(1.5,1.5)(.5,2.7) 
      \bezier180(.5,2.7)(0,4)(4,3.5) 
      \bezier140(4,3.5)(4.5,3.2)(1.5,2.7) 
      \multiput(2,.5)(.4,.05)3{
      \bezier30(0,0)(.2,0)(.2,.3) 
      \bezier30(.2,.3)(.2,.05)(.4,.05) 
      }
      \bezier30(3.2,.65)(3.27,.67)(3.3,.8) 
      \bezier30(3.3,.8)(3.34,.6)(3.53,.66) 
      \bezier30(3.53,.66)(3.6,.73)(3.58,.8) 
      \bezier30(3.58,.8)(3.555,.99)(3.4,.96) 
      \bezier10(2.3,1)(2.3,1.01)(2.31,1.02) 
\end{picture}\hss}}}}

\def\tt{\@ifnextchar_{\top\kern-.5ex}{\top}}
\cdrestoreat

\newtheoremstyle{teorema}{\topsep}{\topsep}
{\thmitem\slshape}{}{\bf}{{\normalfont.}}{.5em}{}
\newtheoremstyle{definizione}{\topsep}{\topsep}
{\thmitem\normalfont}{}{\bf}{{\normalfont.}}{.5em}{}

\def\namefont#1{{\sc #1}}
\swapnumbers
\theoremstyle{teorema}
\newtheorem{theorem}{\namefont{Theorem}}[section]
\newtheorem{lemma}[theorem]{\namefont{Lemma}}
\newtheorem{prop}[theorem]{\namefont{Proposition}}
\newtheorem{cor}[theorem]{\namefont{Corollary}}
\theoremstyle{definizione}
\newtheorem{definition}[theorem]{\namefont{Definition}}
\newtheorem{remark}[theorem]{\namefont{Remark}}

\newtheorem{exms}[theorem]{\namefont{Examples}}

\def\dfn#1{{\bfseries\itshape #1\/}}

\def\cat#1{\ensuremath{\mathcal{#1}}}
\def\Cat#1{\ensuremath{{\normalfont\textsf{\bfseries #1}}}}
\def\id#1{\ensuremath{\mathrm{id}_{#1}}}

\def\op{{}^{\textrm{\scriptsize op}}}
\def\op{^{\textrm{\scriptsize op}}}
\def\op{\strut^{\textrm{\scriptsize op}}}
\def\exl{_{\textrm{\scriptsize ex}}}
\def\exr{_{\textrm{\scriptsize ex/reg}}}
\def\reg{_{\textrm{\scriptsize reg}}}
\def\Gr(#1){\ensuremath{\mathcal{G}_{#1}}}
\def\des#1{\ensuremath{\mathcal{D}\kern-.3ex\textit{es\kern.2ex}_{#1}}}

\let\Land\wedge

\let\ForalL\forall \def\Forall#1.{\ForalL_{#1}}
\let\ExistS\exists \def\Exists#1.{\ExistS_{#1}}
\def\LGE{\mathrel{\bgroup\ooalign{\hfil\raise.8ex\hbox{$\lt$}\hfil
 \crcr\hfil\raise-.5ex\hbox{$\eqslantgtr$}}\egroup}}
\def\Lge{\mathrel{\bgroup
 \ooalign{\hfil\raise.4ex\hbox{$\scriptscriptstyle\lt$}\hfil
 \crcr\hfil\raise-.3ex\hbox{$\scriptscriptstyle\eqslantgtr$}}\egroup}}

\def\gel{=}

\def\pr{\mathrm{pr}}
\def\cmp#1{\ensuremath{\{\kern-2.5pt|{#1}|\kern-2.5pt\}}}
\def\EH{\Cat{ED}\xspace}

\def\QH{\Cat{QED}\xspace}
\def\QX{\Cat{CED}\xspace}
\def\Qex#1#2{\relax\ifx#2q\ensuremath{\left[{#1}\right]}\fi
\ifx#2c\ensuremath{{\left|{#1}\right|}}\fi
\ifx#2x\ensuremath{\widehat{#1}}\fi}
\def\QEx#1#2{\ensuremath{{({#1})_{_{\textrm{\scriptsize #2}}}}}}

\def\Q#1{\Qex{#1}{q}}
\def\X#1{\QEX{#1}{x}}
\def\P#1{\Qex{#1}{c}}
\def\Q#1{\QEx{#1}{q}}
\def\X#1{\QEx{#1}{x}}
\def\P#1{\QEx{#1}{c}}
\def\R#1{\QEx{#1}{r}}
\def\D{{\rotatebox[origin=c]{180}{\ensuremath{E}}\kern-.3ex}}
\def\B{{\rotatebox[origin=c]{180}{\ensuremath{A}}\kern-.6ex}}
\def\QD{\Q{{\rotatebox[origin=c]{180}{\ensuremath{E}}}}\kern-.3ex}

\def\EC#1{\ensuremath
 {\left\lfloor\left.\kern-.2ex{#1}\kern-.2ex\right\rceil\right.}}
\def\Strut{\hbox{\vrule height.75em depth.35em width0pt}}
\def\ec#1{\ensuremath
 {\left\lfloor\left.\kern-.3ex\Strut{#1}\Strut\kern-.3ex\right\rceil\right.}}
\def\ec#1{\ensuremath{\left[{#1}\right]}}

\def\dpd{\mathchoice{\textstyle\prod}{\prod}{\prod}{\prod}\kern-.2ex\strut}

\def\prp{\kern.5ex\mbox{\normalfont\it prop}\kern.5ex}
\def\prps{\kern.5ex\mbox{\normalfont\it prop}_s\kern.5ex}

\def\ie{{\textit{i.e.}}\xspace}
\def\eg{{\textit{e.g.}}\xspace}
\def\loccit{{\textit{loc.cit.}}\xspace}

\long\def\beginskip#1\endskip{\newpage}

\date{}
\begin{document}

\title{
Elementary quotient completion
}
\author{Maria Emilia Maietti\thanks{%
Dipartimento di Matematica Pura ed Applicata,
Universit\`a degli Studi di Padova,
via Trieste 63, 35121 Padova, Italy,
email:~\texttt{maietti@math.unipd.it}}
\and
Giuseppe Rosolini\thanks{%
Dipartimento di Matematica, Universit\`a degli Studi di Genova,
via Dodecaneso 35, 16146 Genova, Italy,
email:~\texttt{rosolini@unige.it}}}
\maketitle

\begin{abstract}
We extend the notion of exact completion on a weakly lex category
to elementary doctrines. We show how any such
doctrine admits an elementary quotient completion, which freely adds
effective quotients and extensional
equality. We note that the elementary quotient completion can be
obtained as the composite of two free constructions: one adds
effective quotients, and the other forces extensionality of maps.
We also prove that each construction preserves comprehensions. 
\end{abstract}

{\bf MSC 2000}: 03G30 03B15 18C50 03B20 03F55

{\bf Keywords:}
quotient completion, split fibration, free construction

\section{Introduction}
Constructions for completing a category by quotients has been widely
studied in category theory. The main instance is the so-called exact
completion in \cite{CarboniA:freecl,CarboniA:regec}
which shows how to add, in a finitary way, quotients
that are defined as effective coequalizers of monic equivalence
relations to suitable categories by turning them into exact categories.

The use of quotient completion is also pervasive in interactive
theorem proving where proofs are performed in appropriate systems of
formalized mathematics in a computer-assisted way.
Indeed the use of a quotient completion is rather compulsory when
mathematics is formalized within an intensional 
type theory, such as the Calculus of (Co)In\-duc\-tive
Constructions~\cite{tc90,CP89} or 
Martin-L{\"o}f's type theory~\cite{PMTT}.
In such a context, the abstract construction of quotient completion
provides a formal framework where to combine the usual practice of
(extensional) mathematics, with the need of formalizing it in an
intensional theory with strong decidable properties (such as decidable
type-checking) on which to perform the extraction of algorithmic
contents from proofs. 

To make explicit 
the use of quotient completion in the formalization of constructive
mathematics, 
in \cite{m09} it has been included as a part of the definition of
constructive foundation. According to \cite{m09}, a constructive
foundation must be a two-level theory as first argued in \cite{mtt}:
it must be equipped with  an intensional level, which can be
represented by a suitable 
starting category \cat{C}, and an extensional level that can be seen
as (a fragment of) the internal language of a suitable quotient
completion of \cat{C}. 
As investigated in \cite{MaiettiME:quofcm}, some examples of quotient
completion 
performed on intensional theories, such as the intensional level of
the minimalist foundation in \cite{m09}, or the Calculus of
Constructions, do not fall under the known constructions of exact
completion given that the corresponding type theoretic categories
closed under quotients are not exact. 

In \cite{MaiettiME:quofcm} we studied the abstract categorical
structure behind such quotient completions. To this purpose we
introduced the notion of equivalence relation and quotient relative  
to a suitable fibered poset and produced a free construction adding
effective quotients---hence the name elementary quotient
completion---to elementary doctrines.

In the present paper we isolate the basic components of the free
constructions in \cite{MaiettiME:quofcm}. After recalling the basic
notions required in the sequel, we show how to add
effective quotients freely to an elementary 
doctrine in the sense of \cite{LawvereF:equhcs}, a fibered
infsemilattice on a cartesian category, endowed 
with equality. Separately, we describe how to force
extensional equality of maps to (the base of) an elementary
doctrine. Then we prove that the two constructions can be combined to 
give the elementary quotient completion. Finally we check that the
exact completion of a weakly lex cartesian category is an instance of
the elementary quotient completion while the regular completion
of a weakly lex cartesian category is an instance of a rather
different construction.

\section{Doctrines}
The notion of a doctrine is the basic categorical concept
we adopt to analyse quotients.
It was introduced, in a series of seminal papers, by F.W.~Lawvere to
synthetize the structural properties of logical systems, see
\cite{LawvereF:adjif,LawvereF:diaacc,LawvereF:equhcs}, see also
\cite{LawvereF:setfm} for a unified survey. Lawvere's
crucial intuition was to consider logical languages and theories as
fibrations to study their 2-categorical properties, \eg connectives
and quantifiers are determined by structural adjunctions. That
approach proved extremely fruitful, see 
\cite{MakkaiM:firocl,LambekJ:inthoc,JacobsB:catltt,TaylorP:prafom,OostenJ:reaait} 
and the references therein.

Taking advantage of the algebraic presentation of logic by fibrations,
we first introduce a general notion of elementary doctrine which we
found appropriate to study the notion of quotient of an equivalence
relation, see \cite{MaiettiME:quofcm}.

\begin{definition}\label{eld}
An \dfn{elementary doctrine} is a functor 
$P:\cat{C}\op\longrightarrow\Cat{InfSL}$
from (the opposite of) a category \cat{C} with binary products
to the category of inf-semilattices and homomorphisms
such that, for every object $A$ in \cat{C}, there is an
object $\delta_A$ in $P(A\times A)$ and
\begin{enumerate}
\item the assignment
$$\D_{<\id{A},\id{A}>}(\alpha)\colon=
P_{\pr_1}(\alpha)\Land_{A\times A}\delta_A$$
for $\alpha$ in $P(A)$ determines a left adjoint to 
$P_{<\id{A},\id{A}>}:P(A\times A)\to P(A)$---the action of a doctrine
$P$ on an arrow is written as $P_f$
\item for every map 
$e\colon=<\pr_1,\pr_2,\pr_2>:X\times A\to X\times A\times A$ in \cat{C},
the assignment
$$\D_{e}(\alpha)\colon=
P_{<\pr_1,\pr_2>}(\alpha)\Land_{A\times A}P_{<\pr_2,\pr_3>}(\delta_A)$$
for $\alpha$ in $P(X\times A)$ determines a left adjoint to 
$P_{e}:P(X\times A\times A)\to P(X\times A)$.
\end{enumerate}
\end{definition}

\begin{remark}
\noindent(a) In case \cat{C}
has a terminal object, conditions (ii) entails condition (i).

\noindent(b)
One has that $\tt_A\leq P_{<\id{A},\id{A}>}(\delta_A)$ and 
$\delta_A\leq P_{f\times f}(\delta_B)$ for $f:A\to B$.
\end{remark}

\begin{remark}
For $\alpha_1$ in $P(X_1\times Y_1)$ and $\alpha_2$ in $P(X_2\times
Y_2)$, it is useful to introduce a notation 
like $\alpha_1\boxtimes\alpha_2$ for the object  
$$P_{<\pr_1,\pr_3>}(\alpha_1)\Land P_{<\pr_2,\pr_4>}(\alpha_2)$$
in 
$P(X_1\times X_2\times Y_1\times Y_2)$ 
where $\pr_i, i=1,2,3,4$, are the projections from 
$X_1\times X_2\times Y_1\times Y_2$ to each of the four factors.

Condition \ref{eld}(ii) is to request that
$\delta_{A\times B}=\delta_A\boxtimes\delta_B$ for every pair of
objects $A$ and $B$ in \cat{C}.
\end{remark}

\begin{exms}\label{ltae}
\noindent(a) 
The standard example of an
indexed poset is the fibration of 
subobjects. Consider a category \cat{X} with
products and pullbacks.
The functor $S:\cat{X}\op\longrightarrow\Cat{InfSL}$ assigns to any object
$A$ in \cat{X} the poset $S(A)$ of subobjects of $A$ in \cat{X}.
For an arrow $f:B\to A$, the assignment 
that maps a subobject in $S(A)$ to that represented
by the left-hand arrow in any pullback along $f$ of its
produces a functor $S_f:S(A)\to S(B)$ that preserves products.

The elementary structure is provided by the diagonal maps.

\noindent(b)
The leading logical example is the indexed order
$LT:\cat{V}\op\longrightarrow\Cat{InfSL}$ given by the
Lindenbaum-Tarski algebras of well-formed formulae of a theory
$\mathscr{T}$ with equality in a first order language
$\mathscr{L}$. 

The domain category is the category \cat{V} of lists of variables and
term substitutions:
\begin{description}
\item[object of \cat{V}] are lists
of distinct variables $\vec x=(x_1,\ldots,x_n)$
\item[arrows] are lists of substitutions\footnote{We shall
employ a vector notation for lists of terms in the language as well as
for simultaneous substitutions such as $[\vec t/\vec y]$ in place of
$[t_1/y_1,\ldots,t_m/y_m]$.}
for variables $[\vec t/\vec y]:\vec x\to \vec y$ where each term 
$t_j$ in $\vec t$ is built in $\mathscr{L}$ on the variables
$x_1,\ldots,x_n$
\item[composition] 
$\xymatrix@1@=4em{\vec x\ar[r]^{[\vec t/\vec y]}&
\vec y\ar[r]^{[\vec s/\vec z]}&\vec z}$
is given by simultaneous substitutions
$$\xymatrix@=10em{\vec x
\ar[r]^{\left[s_1[\vec t/\vec y]/z_1,\ldots,s_k[\vec t/\vec y]/z_k\right]}
&\vec z}$$
\end{description}
The product of two objects $\vec x$ and $\vec y$ is given by a(ny)
list $\vec w$ of as many distinct variables as the sum of the number
of variables in $\vec x$ and of that in $\vec y$. Projections are
given by substitution of the variables in $\vec x$ with the
first in $\vec w$ and of the variables in $\vec y$ with the last in
$\vec w$.

The functor $LT:\cat{V}\op\longrightarrow\Cat{InfSL}$ is given
as follows: for a list of distinct variables $\vec x$,
the category $LT(\vec x)$ has
\begin{description}
\item[objects] equivalence classes of well-formed formulae of
$\mathscr{L}$ with no more free variables than $x_1$,\ldots,$x_n$ with
respect to provable reciprocal consequence 
$W\dashv\vdash_{\mathscr T}W'$ in $\mathscr T$.
\item[arrows] $\ec{W\kern.2ex}\to\ec{V\kern.2ex}$ are
the provable consequences $W\vdash_{\mathscr T}V$ in $\mathscr T$ for
some pair of representatives (hence for any pair)
\item[composition] is given by the cut rule in the logical calculus
\item[identities] $\ec{W\kern.2ex}\to\ec{W\kern.2ex}$
are given by the logical rules $W\vdash_{\mathscr T}W$
\end{description}
For a list of distinct variables $\vec x$, the category
$LT(\vec x)$ has finite limits: a terminal object is 
$\vec x=\vec x$ and products are given by conjunctions of
formulae.

\noindent(c) 
Consider a cartesian category \cat{S} with 
weak pullbacks.
Another example of elementary doctrine which appears {\it prima facie}
very similar to previous example (a) is given by
the functor of {\it weak subobjects}
$\Psi:\cat{S}\op\longrightarrow\Cat{InfSL}$
which evaluates as the poset reflection of each comma category
$\cat{S}/A$ at each object $A$ of \cat{S}, introduced in
\cite{GrandisM:weasec}.

The apparently minor difference between the present example
and example (a) depends though on the possibility of factoring an
arbitrary arrow as a retraction followed 
by a monomorphism: for instance this can be achieved in the category
\Cat{Set} of sets and functions thanks to the Axiom of Choice, see
\loccit
\end{exms}

It is possible to express precisely how the examples are related once
we consider the 2-category \EH of elementary doctrines:
\begin{description}
\item[the 1-arrows] are pairs $(F,b)$ 
$$
\xymatrix@C=4em@R=1em{
{\cat{C}\op}\ar[rd]^(.4){P}_(.4){}="P"\ar[dd]_F&\\
           & {\Cat{InfSL}}\\
{\cat{D}\op}\ar[ru]_(.4){R}^(.4){}="R"&\ar"P";"R"_b^{\kern-.4ex\cdot}}
$$
where the functor $F$ preserves products and, for every object $A$ in
\cat{C}, the functor $b_A:P(A)\to R(F(A))$ preserves all the
structure. More explicitly, $b_A$ preserves finite meets and, for
every object $A$ in \cat{C}, 
\begin{equation}\label{two}
b_{A\times A}(\delta_A)\gel R_{<F(\pr_1),F(\pr_2)>}(\delta_{F(A)}).
\end{equation}
\item[the 2-arrows] are natural transformations $\theta$ such that
$$
\xymatrix@C=6em@R=1em{
{\cat{C}\op}\ar[rd]^(.4){P}_(.4){}="P"
\ar@<-1ex>@/_/[dd]_F^{}="F"\ar@<1ex>@/^/[dd]^G_{}="G"&\\
           & {\Cat{InfSL}}\\
{\cat{D}\op}\ar[ru]_(.4){R}^(.4){}="R"&
\ar@/_/"P";"R"_{b\kern.5ex\cdot\kern-.5ex}="b"
\ar@<1ex>@/^/"P";"R"^{\kern-.5ex\cdot\kern.5ex c}="c"
\ar"F";"G"^{.}_\theta\ar@{}"b";"c"|{\leq}}
$$
so that, for every object $A$ in \cat{C} and every $\alpha$ in $P(A)$,
one has $R_{\theta_A}(b_A(\alpha))\leq c_A(\alpha)$.
\end{description}

\begin{exms}\label{tae}
(a) Given a theory $\mathscr{T}$ with equality in a first order
language, a 1-arrow $(F,b):LT\to S$ from the elementary doctrine
$LT:\cat{V}\op\longrightarrow\Cat{InfSL}$ as in \ref{ltae}(a) into
the elementary doctrine $S:\Cat{Set}\op\longrightarrow\Cat{InfSL}$ as in
\ref{ltae}(b) determines a model $\mathfrak{M}$ of $\mathscr{T}$ where
the set underlying the intepretation is $F(x=x)$. In fact, there is an
equivalence between the category $\EH(LT,S)$ and the category of
models of $\mathscr{T}$ and homomorphisms.

\noindent(b)
Given a category \cat{X} with
products and pullbacks, one can consider the two indexed posets:
that of subobjects $S:\cat{X}\op\longrightarrow\Cat{InfSL}$, and the other 
$\Psi:\cat{X}\op\longrightarrow\Cat{InfSL}$, obtained by the poset reflection
of each comma category $\cat{X}/A$, for $A$ in \cat{X}. The inclusions
of the poset $S(A)$ of subobjects over $A$ into the poset reflection
of  $\cat{X}/A$ extend to a 1-arrow from $S$ to $\Psi$ which is an
equivalence exactly when every arrow in \cat{X} can be factored as a
retraction followed by a monic.
\end{exms}

\section{Quotients in an elementary doctrine}

The structure of elementary doctrine is suitable to describe the
notions of an equivalence relation and of a quotient for such a
relation.

\begin{definition}\label{per}
Given an elementary doctrine
$P:\cat{C}\op\longrightarrow\Cat{InfSL}$, an object $A$ in
\cat{C} and an object $\rho$ in $P(A\times A)$, we say that 
$\rho$ is a \dfn{$P$-equivalence relation on $A$} if
it satisfies
\begin{description}
\item[\dfn{reflexivity}:] $\delta_A\leq\rho$
\item[\dfn{symmetry}:]
$\rho\leq P_{<\pr_2,\pr_1>}(\rho)$, for $\pr_1,\pr_2:A\times A\to A$
the first and second projection, respectively
\item[\dfn{transitivity}:]
$P_{<\pr_1,\pr_2>}(\rho)\Land P_{<\pr_2,\pr_3>}(\rho)\leq
P_{<\pr_1,\pr_3>}(\rho)$, for $\pr_1,\pr_2,\pr_3:A\times A\times A\to A$
the projections to the first, second and third factor,
respectively.
\end{description}
\end{definition}

In elementary doctrines as those presented in \ref{ltae},
$P$-equivalence relations concide with the usual notion for those of
the form (a) or (b); more interestingly, in cases like (c)
a $\Psi$-equivalence relation is a pseudo-equivalence relation in
\cat{S} in the sense of \cite{CarboniA:freecl}.

For $P:\cat{C}\op\longrightarrow\Cat{InfSL}$
an elementary doctrine, the object $\delta_A$ is a $P$-equivalence
relation on $A$. And for an arrow $f:A\to B$ in \cat{C},
the functor $P_{f\times f}:P(B\times B)\to P(A\times A)$ takes a
$P$-equivalence relation $\sigma$ on $B$ to a $P$-equivalence relation
on $A$. Hence, the
\dfn{$P$-kernel of $f:A\to B$}, the object $P_{f\times f}(\delta_B)$
of $P_{A\times A}$ is a $P$-equivalence relation on $A$. In such a
case, one speaks of $P_{f\times f}(\delta_B)$ as an \dfn{effective}
$P$-equivalence relation.

\begin{remark}
A 1-arrow $(F,b):P\to R$ in \EH takes a $P$-equivalence relation on
$A$ to an $R$-equivalence relation on $FA$.
\end{remark}

\begin{definition}\label{quot}
Let $P:\cat{C}\op\longrightarrow\Cat{InfSL}$ be an elementary
doctrine. Let $\rho$ be a $P$-equivalence relation on $A$. 
A \dfn{quotient of $\rho$} is a arrow $q:A\to C$ in \cat{C} such that
$\rho\leq P_{q\times q}(\delta_C)$ 
and, for every arrow $g:A\to Z$ such that
$\rho\leq P_{g\times g}(\delta_Z)$, there is a unique arrow $h:C\to Z$
such that $g=h\circ q$. 

We say that such a quotient is \dfn{stable} if, in every pullback
$$\xymatrix{A'\ar[d]_{f'}\ar[r]^{q'}&C'\ar[d]^{f}\\A\ar[r]_q&C}$$
in \cat{C}, the arrow $q':A'\to C'$ is a quotient.
\end{definition}

\begin{remark}
Note that the inequality $\rho\leq P_{q\times q}(\delta_C)$ in
\ref{quot} becomes an identity exactly when $\rho$ is effective.
\end{remark}

In the elementary doctrine $S:\cat{X}\op\longrightarrow\Cat{InfSL}$ 
obtained from a category \cat{X} with products and pullbacks as in
\ref{ltae}(a), a quotient of the $S$-e\-quiv\-a\-lence relation
$\ec{\smash{\xymatrix@=2.3ex@1{r:R\ \ar@{>->}[r]&A\times A}}}$ is 
precisely a coequalizer of the pair of 
$$\xymatrix@C=4em{R\ar@<.5ex>[r]^(.45){\pr_1\circ r}
\ar@<-.5ex>[r]_(.45){\pr_2\circ r}&A}$$
In particular, all $S$-equivalence relations have
stable, effective quotients if and only if the category \cat{C} is
exact.

Similarly, in the elementary doctrine
$\Psi:\cat{S}\op\longrightarrow\Cat{InfSL}$ 
obtained from a cartesian category \cat{X} with weak pullbacks as in
\ref{ltae}(c), a quotient of the $\Psi$-e\-quiv\-a\-lence relation 
$\ec{\smash{\xymatrix@=2.3ex@1{r:R\ar[r]&A\times A}}}$ is 
precisely a coequalizer of the pair of 
$$\xymatrix@C=4em{R\ar@<.5ex>[r]^(.45){\pr_1\circ r}
\ar@<-.5ex>[r]_(.45){\pr_2\circ r}&A}$$
In particular, all $\Psi$-equivalence relations have quotients which
are stable if and only if the category \cat{C} is exact.

\begin{definition}
Given an elementary doctrine
$P:\cat{C}\op\longrightarrow\Cat{InfSL}$ and 
a $P$-e\-quiv\-a\-lence relation $\rho$ on an object $A$ in \cat{C},
the poset of descent data \des{\rho} is the sub-poset of
$P(A)$ on those $\alpha$ such that
$$P_{\pr_1}(\alpha)\Land_{A\times A}\rho\leq P_{\pr_2}(\alpha),$$
where $\pr_1,\pr_2:A\times A\to A$ are the projections.
\end{definition}

\begin{remark}
Given an elementary doctrine
$P:\cat{C}\op\longrightarrow\Cat{InfSL}$,
for $f:A\to B$ in \cat{C}, let $\chi$ be the $P$-kernel 
${P_{f\times f}(\delta_B)}$.
The functor $P_f:P(B)\to P(A)$ applies $P(B)$ into \des\chi.
\end{remark}

\begin{definition}
Given an elementary doctrine
$P:\cat{C}\op\longrightarrow\Cat{InfSL}$ and 
an arrow $f:A\to B$ in \cat{C}, let $\chi$ be the $P$-kernel 
${P_{f\times f}(\delta_B)}$. 
The arrow $f$ is \dfn{descent} if the (obviously faithful) functor
$P_f:P(B)\to\des\chi$ is also full.
The arrow $f$ is \dfn{effective descent} if the functor
$P_f:P(B)\to\des\chi$ is an equivalence.
\end{definition}

Consider the 2-full 2-subcategory \QH of \EH whose objects are
elementary doctrines $P:\cat{C}\op\longrightarrow\Cat{InfSL}$ 
with descent quotients of $P$-equivalence relations.
\begin{description}
\item[The 1-arrows] are those pairs $(F,b)$ in \EH
$$
\xymatrix@C=4em@R=1em{
{\cat{C}\op}\ar[rd]^(.4){P}_(.4){}="P"\ar[dd]_F&\\
           & {\Cat{InfSL}}\\
{\cat{D}\op}\ar[ru]_(.4){R}^(.4){}="R"&\ar"P";"R"_b^{\kern-.4ex\cdot}}
$$
such that $F$ preserves quotients in the sense that, if $q:A\to C$ is 
a quotient of a $P$-equivalence relation $\rho$ on
$A$, then $Fq:FA\to FC$ is a quotient of the $R$-equivalence relation
$R_{<F(\pr_1),F(\pr_2)>}(b_{A\times A}(\rho))$ on $FA$.
\end{description}

\section{Completing with quotients as a free construction}\label{main}

It is a simple construction that produces an elementary
doctrine with quotients. We shall present it in the following and
prove that it satisfies a universal property.

Let $P:\cat{C}\op\longrightarrow\Cat{InfSL}$ be an elementary
doctrine for the rest of the section.
Consider the category $\cat{R}_P$ of ``equivalence relations of $P$'': 
\begin{description}
\item[an object of $\cat{R}_P$] is a pair $(A,\rho)$ such that $\rho$
is a $P$-equivalence relation on $A$
\item[an arrow {$f:(A,\rho)\to(B,\sigma)$}] is an 
arrow $f:A\to B$ in \cat{C} such that 
$\rho\leq_{A\times A}P_{f\times f}(\sigma)$ in $P(A\times A)$.
\end{description}
Composition is given by that of \cat{C}, and
identities are the identities of \cat{C}.

The indexed poset
$\Q{P}:\cat{R}_P\op\longrightarrow\Cat{InfSL}$
on $\cat{R}_P$ will be given by categories of descent data: 
on an object $(A,\rho)$ it is defined as
$$
\Q{P}(A,\rho)\colon=\des{\rho}
$$
and the following lemma is instrumental to give the assignment on
arrows using the action of $P$ on arrows.

\begin{lemma}
With the notation used above, let $(A,\rho)$ and $(B,\sigma)$ be
objects in $\cat{R}_P$, and let $\beta$ be an object in
\des{\sigma}. If $f:(A,\rho)\to(B,\sigma)$ is an arrow in 
$\cat{R}_P$, then
$P_f(\beta)$ is in \des{\rho}.
\end{lemma}

\begin{proof}
Since $\beta$ is in \des{\sigma}, it is
$$P_{\pr_1'}(\beta)\Land\sigma\leq_{B\times B}P_{\pr_2'}(\beta)$$
where $\pr_1',\pr_2':B\times B\to B$ are the two projections.
Hence
$$P_{f\times f}(P_{\pr_1}(\beta))\Land
P_{f\times f}(\sigma)\leq_{A\times A}P_{f\times f}(P_{\pr_2}(\beta))$$
as $P_{f\times f}$ preserves the structure. Since
$\rho\leq_{A\times A}P_{f\times f}(\sigma)$,
$$P_{\pr_1}(P_{f}(\beta))\Land\rho\leq_{A\times A}P_{\pr_2}(P_{f}(\beta))$$
where $\pr_1,\pr_2:A\times A\to A$ are the two projections.\end{proof}

\begin{lemma}
With the notation used above,
$\Q{P}:\cat{R}_P\op\longrightarrow\Cat{InfSL}$ is an elementary
doctrine.
\end{lemma}

\begin{proof}
For $(A,\rho)$ and $(B,\sigma)$ in $\cat{R}_P$
let 
$\pr_1,\pr_3:A\times B\times A\times B\to A$ and
$\pr_2,\pr_4:A\times B\times A\times B\to B$ be the four
projections. As a meet of two $P$-equivalence relations on 
$A\times B$, the $P$-equivalence relation
$$\rho\boxtimes\sigma\colon=P_{<\pr_1,\pr_3>}(\rho)
\Land_{A\times B\times A\times B} P_{<\pr_2,\pr_4>}(\sigma)$$
provides an object $(A\times B,\rho\boxtimes\sigma)$ in $\cat{R}_P$
which, together with the arrows determined by the two projections from
$A\times B$, is a product of $(A,\rho)$ and $(B,\sigma)$ in
$\cat{R}_P$.\newline
For each $(A,\rho)$, the sub-poset $\des\rho\subseteq P(A)$ is
closed under finite meets.\newline
For an object $(A,\rho)$ in $\cat{R}_P$, consider the object
$P_{<\pr_1,\pr_2>}(\rho)$ in $P(A\times A\times A\times A)$. It is
easy to see that it is in $\des{\rho\boxtimes\rho}$. Such objects
satisfy \ref{eld}~(i) and (ii):
the assignment
$(\QD)_{\Delta_A}(\alpha)\colon=P_{\pr_1}(\alpha)\Land_{A\times A}\rho$,
for $\alpha$ in \des\rho, gives the left adjoint
$(\QD)_{\Delta_A}$ for $(\Q{P})_{\Delta_A}$.
Indeed, let $\theta$ be in \des{\rho\boxtimes\rho} such that
$\alpha\leq_{(A,\rho)}(\Q{P})_{\Delta_A}(\theta)$, 
\ie
$\alpha\leq_AP_{\Delta_A}(\theta)$. Thus
$\D_{\Delta_A}(\alpha)\leq_{A}\theta$ and one has
$$\begin{array}{r@{}l}
P_{\pr_1}(\alpha)\Land P_{<\pr_1,\pr_2>}(\delta_A)\Land
P_{<\pr_2,\pr_3>}(\rho)&{}\leq_{A\times A\times A}
P_{<\pr_1,\pr_2>}(\theta)\Land P_{<\pr_2,\pr_3>}(\rho)\\[1.5ex]
&{}\leq_{A\times A\times A}P_{<\pr_1,\pr_3>}(\theta)
\end{array}$$
for $\pr_i:A\times A\times A\to A,\quad i=1,2,3$, the
projections.
Hence $P_{\pr_1}(\alpha)\Land\rho\leq_{A\times A}\theta$, \ie
$(\QD)_{\Delta_A}(\alpha)
\leq_{(A\times A,\rho\boxtimes\rho)}\theta$. It is easy 
to prove the converse that, if
$(\QD)_{\Delta_A}(\alpha)\leq\theta$, then
$\alpha\leq(\Q{P})_{\Delta_A}(\theta)$. The proof of condition
\ref{eld}(ii) is similar.\end{proof}

There is an obvious 1-arrow $(J,j):P\to\Q{P}$ in \EH, where
$J:\cat{C}\op\longrightarrow\cat{R}_P$ sends an object $A$ in \cat{C} to
$(A,\delta_A)$ and an arrow $f:A\to B$ to 
$f:(A,\delta_A)\to(B,\delta_B)$ since 
$\delta_A\leq_{A\times A}P_{f\times f}(\delta_B)$, and
$j_A:P(A)\to\Q{P}(A,\delta_A)$ is the identity since, by definition,
$$\Q{P}(A,\delta_A)=\des{\delta_A}=P(A).$$
It is immediate to see that $J$ is full and faithful and that $(J,j)$
is a change of base.

\begin{remark}\label{cov}
Note that an object of the form $(A,\delta_A)$ in $\cat{R}_P$ is
projective with respect to quotients of \Q{P}-e\-quiv\-a\-lence
relation, and that every object in $\cat{R}_P$ is a
quotient of a \Q{P}-equivalence relation on such a projective.
\end{remark}

\begin{lemma}
With the notation used above,
$\Q{P}:\cat{R}_P\op\longrightarrow\Cat{InfSL}$ has descent
quotients of $\Q{P}$-equivalence relations. Moreover, quotients are
stable and effective descent, and $P$-equivalence relations are
effective.
\end{lemma}

\begin{proof}
Since the sub-poset $\des\rho\subseteq P(A)$ is
closed under finite meets, a \Q{P}-equivalence relation $\tau$ on
$(A,\rho)$ is also a $P$-equivalence relation on $A$. It is easy to see
that $\id{A}:(A,\rho)\to(A,\tau)$ is a descent quotient since
$\rho\leq_{A\times A}\tau$---actually, effectively so. It follows
immediately that $\tau$ is the $P$-kernel of the quotient
$\id{A}:(A,\rho)\to(A,\tau)$. To see that it is
also stable, suppose
$$\xymatrix@=3em{(B,\upsilon)\ar[r]_{f'}\ar[d]^{g}
&(A,\rho)\ar[d]_{\id{A}}\\
(C,\sigma)\ar[r]^{f}&(A,\tau)}$$
is a pullback in $\cat{R}_P$. So in the commutative diagram
$$\xymatrix@=3em{
(C,\delta_C)\ar[rd]_{\id{C}}\ar@/^20pt/[rr]^(.3){f}&
(B,\upsilon)\ar[r]_{f'}\ar[d]^{g}
&(A,\rho)\ar[d]_{\id{A}}\\
&(C,\sigma)\ar[r]^{f}&(A,\tau)}$$
there is a fill-in map $h:(C,\delta_C)\to(B,\upsilon)$. It is now
easy to see that $g:(B,\upsilon)\to(C,\sigma)$ is a
quotient.\end{proof}

We can now prove that there is a left bi-adjoint to the forgetful
2-functor $U:\QH\to\EH$.

\begin{theorem}\label{mthm}
For every elementary doctrine
$P:\cat{C}\op\longrightarrow\Cat{InfSL}$, pre-composition with the
1-arrow
$$
\xymatrix@C=4em@R=1em{
{\cat{C}\op}\ar[rd]^(.4){P}_(.4){}="P"\ar[dd]_J&\\
           & {\Cat{InfSL}}\\
{\cat{R}_P\op}\ar[ru]_(.4){\Q{P}}^(.4){}="R"&\ar"P";"R"_j^{\kern-.4ex\cdot}}
$$
in \EH induces an essential equivalence of categories 
\begin{equation}\label{eqv}
-\circ(J,j):\QH(\Q{P},Z)\equiv\EH(P,Z)
\end{equation}
for every $Z$ in \QH.
\end{theorem}

\begin{proof}
Suppose $Z$ is a doctrine in \QH.
As to full faithfulness of the functor in (\ref{eqv}), consider two
pairs $(F,b)$ and $(G,c)$ of 1-arrows from \Q{P} to $Z$. By \ref{cov},
the natural transformation $\theta:F\stackrel.\to G$ in 
a 2-arrow from $(F,b)$ to $(G,c)$ in \QH is completely determined by
its action on objects in the image of $J$ and $\Q{P}$-equivalence
relations on these. And, since a quotient $q:U\to V$ of an
$Z$-equivalence relation $r$ on $U$ is descent, $Z(V)$ is a full
sub-poset of $Z(U)$. Thus essential surjectivity of the
functor in (\ref{eqv}) follows from \ref{cov}.\end{proof}

Recall that, for an elementary doctrine
$P:\cat{C}\op\longrightarrow\Cat{InfSL}$, and for an
object $\alpha$ in some $P(A)$, a \dfn{comprehensions} of $\alpha$ is
a map $\cmp\alpha:X\to A$ in \cat{C} such 
that $P_{\cmp\alpha}(\alpha)=\tt_X$ and, for every $f:Z\to A$ such
that $P_f(\alpha)=\tt_Z$ there is a unique map $g:Z\to X$ such that
$f=\cmp\alpha\circ g$. One says that $P$ \dfn{has comprehensions} if
every $\alpha$ has a comprehension, and that $P$ 
\dfn{has full comprehensions} if, moreover, $\alpha\leq\beta$ in
$P(A)$ whenever $\cmp\alpha$ factors through $\cmp\beta$.

\begin{lemma}
Let $P:\cat{C}\op\longrightarrow\Cat{InfSL}$ be an elementary
doctrine.
If $P$ has comprehensions, then \Q{P} has comprehensions. Moreover,
given a comprehension $\cmp\alpha:X\to A$ of $\alpha$ in $P(A)$, the
map $J(\cmp\alpha):JX\to JA$ is a comprehension of $j_A(\alpha)$ if
and only if $\delta_X=P_{\cmp\alpha\times\cmp\alpha}(\delta_A)$.
\end{lemma}

\begin{proof}
Suppose $(A,\rho)$ is in $\cat{R}_P$ and $\alpha$ in
$\Q{P}(A,\rho)=\des\rho\subseteq P(A)$. Let $\cmp\alpha:X\to A$ be a
comprehension in \cat{C} of $\alpha$ as an object of $P(A)$ and
consider the object $(X,P_{\cmp\alpha\times\cmp\alpha}(\rho))$ in
$\cat{R}_P$. Clearly 
$\cmp\alpha:(X,P_{\cmp\alpha\times\cmp\alpha}(\rho))\to(A,\rho)$: we
intend to show 
that that map is a comprehension of $\alpha$ as an object in
$\Q{P}(A,\rho)$. The following is a trivial computation in
$\des{P_{\cmp\alpha\times\cmp\alpha}(\rho)}\subseteq P(X)$:
$$\tt_X\leq P_{\cmp\alpha}(\alpha)=\Q{P}_{\cmp\alpha}(\alpha).$$
Suppose now that $f:(Z,\sigma)\to(A,\rho)$ is such that
$tt_Z\leq\Q{P}_{f}(\alpha)$. Since $\cmp\alpha$ is a comprehension in
\cat{C}, there is a unique map $g:Z\to X$ such that 
$f=\cmp\alpha\circ g$. To conclude, it is enough to show that
$g:(Z,\sigma)\to(X,P_{\cmp\alpha\times\cmp\alpha}(\rho))$, but
$$\sigma\leq P_{f\times f}(\rho)
=P_{g\times g}(P_{\cmp\alpha\times\cmp\alpha}(\rho)).$$
As for the second part of the statement, let $\alpha$ be in $P(A)$ and
let $\cmp\alpha:X\to A$ be a comprehension of $\alpha$ in
\cat{C}. Suppose, first, that
$\delta_X=P_{\cmp\alpha\times\cmp\alpha}(\delta_A)$, and consider a
map $f:(Z,\sigma)\to (A,\delta_A)$ such that
$(\Q{P})_f(\alpha)=\tt_Z$. By definition of $\Q{P}$, there is a unique
map $g:Z\to X$ such $f=\cmp\alpha\circ g$ in \cat{C}. Thus
$$\sigma\leq P_{f\times f}(\delta_A)
=P_{g\times g}P_{\cmp\alpha\times\cmp\alpha}(\delta_A)
=P_{g\times g}(\delta_X).$$
Conversely, suppose
$\cmp\alpha:(X,\delta_X)\to(A,\delta_A)$ in $\cat{R}_P$ is a
comprehension of $\alpha$ in \Q{P}. Consider
$\cmp\alpha:
(X,P_{\cmp\alpha\times\cmp\alpha}(\delta_A))\to(A,\delta_A)$. Since
$(\Q{P})_{\cmp\alpha}(\alpha)=P_{\cmp\alpha}(\alpha)=\tt_X$, the map must
factor through $\cmp\alpha:(X,\delta_X)\to(A,\delta_A)$, necessarily
with the identity map. Hence the conclusion follows.\end{proof}

\begin{remark}
When $P$ has full comprehensions, the condition
$\delta_X=P_{\cmp\alpha\times\cmp\alpha}(\delta_A)$ is ensured for all
$A$ and $\alpha$.
\end{remark}

Recall that the fibration of vertical maps on the category of points
freely adds comprehensions to a given fibration producing an indexed
poset in case the given fibration is such, see
\cite{JacobsB:catltt}. In our case of interest, for a doctrine
$P:\cat{C}\op\longrightarrow\Cat{InfSL}$, the indexed poset 
consists of the base category \Gr(P) where
\begin{description}
\item[an object] is a pair $(A,\alpha)$ where
$A$ is in \cat{C} and $\alpha$ is in $P(A)$
\item[an arrow {$f:(A,\alpha)\to(B,\beta)$}] is an arrow $f:A\to B$ in
\cat{C} such that $\alpha\leq P_f(\beta)$.
\end{description}
The category \Gr(P) has products and there is a natural embedding 
$I:\cat{C}\to\Gr(P)$ which maps $A$ to $(A,\top_A)$. The indexed
functor extends to
$\P{P}:\Gr(P)\op\longrightarrow\Cat{InfSL}$
along $I$ by setting 
$\P{P}(A,\alpha)\colon=\{\gamma\in P(A)\mid
\gamma\leq\alpha\}$. Moreover, the comprehensions in \P{P} are full. 
As an immediate corollary, we have the
following. 

\begin{theorem}\label{cthn}
There is a left bi-adjoint to the forgetful 2-functor from the full
2-category of \QH on elementary doctrines with comprehensions and
descent quotients into the 2-category \EH of elementary doctrines. 
\end{theorem}

\begin{proof}
The left bi-adjoint sends an elementary doctrine
$P:\cat{C}\op\longrightarrow\Cat{InfSL}$ to the elementary
doctrine
$\Q{\P{P}}:\cat{R}_{\P{P}}\op\longrightarrow\Cat{InfSL}$.\end{proof}

\section{Extensional equality}%

In \cite{MaiettiME:quofcm}, ``extensional'' models of constructive
theories, presented as doctrines
$P:\cat{C}\op\longrightarrow\Cat{InfSL}$, 
were obtained by forcing the equality of arrows
$f,g:A\to B$
in the base
category \cat{C} to correspond to the ``provable'' equality
$\tt_A\leq_A P_{<f,g>}(\delta_B)$ in the fibre $P(A)$. We recall from
\cite{JacobsB:catltt} the basic property that supports the notion of very
strong equality for the case of an elementary doctrine.

\begin{prop}
Let $P:\cat{C}\op\longrightarrow\Cat{InfSL}$ be an elementary
doctrine and let $A$ be an object in \cat{C}. The diagonal
$<\id{A},\id{A}>:A\to A\times A$ is a comprehension if and only if it
is the comprehension of $\delta_A$.
\end{prop}

\begin{definition}
Given an elementary doctrine
$P:\cat{C}\op\longrightarrow\Cat{InfSL}$ we say that it has
\dfn{comprehensive diagonals} if every diagonal map 
$<\id{A},\id{A}>:A\to A\times A$ is a comprehension.
\end{definition}

\begin{remark}\label{remce}
In case \cat{C} has equalizers, one finds that $P$ has comprehensive
diagonals in the sense of \cite{MaiettiME:quofcm}.
\end{remark}

Let $P:\cat{C}\op\longrightarrow\Cat{InfSL}$ be an elementary
doctrine for the rest of the section.
Consider the category $\cat{X}_P$, the ``extensional collapse'' of
$P$:
\begin{description}
\item[the objects of $\cat{R}_P$] are the objects of \cat{C}
\item[an arrow {$\ec{f}:A\to B$}] is an equivalence class
of arrows $f:A\to B$ in \cat{C} such that 
$\delta_A\leq_{A\times A}P_{f\times f}(\delta_B)$ in $P(A\times A)$ with
respect to the equivalence which relates $f$ and $f'$ when
$\delta_A\leq_{A\times A}P_{f\times f'}(\delta_B)$.
\end{description}
Composition is given by that of \cat{C} on representatives, and
identities are represented by identities of \cat{C}.

The indexed inf-semilattice
$\X{P}:\cat{X}_P\op\longrightarrow\Cat{InfSL}$
on $\cat{X}_P$ will be given essentially by $P$ itself; the following
lemma is instrumental to give the assignment on 
arrows using the action of $P$ on arrows.

\begin{lemma}
With the notation used above, let
$f,g:A\to B$ be arrows in \cat{C} and $\beta$ an object in $P(B)$. 
If $\delta_A\leq_{A\times A}P_{f\times g}(\delta_B)$, then
$P_f(\beta)\gel P_g(\beta)$.
\end{lemma}

\begin{proof}
Since $P$ is elementary,
$$P_{\pr_1'}(\beta)\Land\delta_B\leq_{B\times B}P_{\pr_2'}(\beta)$$
where $\pr_1',\pr_2':B\times B\to B$ are the two projections.
Hence
$$P_{f\times g}(P_{\pr_1}(\beta))\Land
P_{f\times g}(\sigma)\leq_{A\times A}P_{f\times g}(P_{\pr_2}(\beta))$$
and, by the
hypothesis that $\delta_A\leq_{A\times A}P_{f\times g}(\delta_B)$,
$$P_{f\circ\pr_1}(\beta)\Land\delta_A\leq_{A\times A}P_{g\circ\pr_2}(\beta)$$
where $\pr_1,\pr_2:A\times A\to A$ are the two projections. Taking
$P_{\Delta_A}$ of both sides,
$$P_f(\beta)\gel P_f(\beta)\Land\tt_A\gel
P_{\Delta_A}(P_{f\circ\pr_1}(\beta))\Land P_{\Delta_A}(\delta_A)\leq
P_{\Delta_A}(P_{g\circ\pr_2}(\beta))\gel P_g(\beta).$$
The other direction follows by symmetry.\end{proof}

In other words, the elementary doctrine
$P:\cat{C}\op\longrightarrow\Cat{InfSL}$ factors through
the quotient functor $K:\cat{C}\op\longrightarrow\cat{X}_P$. That
induces a 1-arrow 
of \EH from $(K,k):P\to\X{P}$ in \EH, where $k_A$ is the identity for
$A$ in \cat{C}.

Consider the full 2-subcategory \QX of \EH whose objects are
elementary doctrines $P:\cat{C}\op\longrightarrow\Cat{InfSL}$ 
with comprehensive diagonals.

The following result is now obvious.

\begin{lemma}
With the notation used above, 
$\X{P}:\cat{X}_P\op\longrightarrow\Cat{InfSL}$ is an elementary
doctrine with comprehensive diagonals.
\end{lemma}

Also the following is easy.

\begin{theorem}\label{mthn}
For every elementary doctrine
$P:\cat{C}\op\longrightarrow\Cat{InfSL}$, pre-composition with the
1-arrow
$$
\xymatrix@C=4em@R=1em{
{\cat{C}\op}\ar[rd]^(.4){P}_(.4){}="P"\ar[dd]_K&\\
           & {\Cat{InfSL}}\\
{\cat{X}_P\op}\ar[ru]_(.4){\X{P}}^(.4){}="R"&\ar"P";"R"_k^{\kern-.4ex\cdot}}
$$
in \EH induces an essential equivalence of categories 
\begin{equation}\label{eqva}
-\circ(K,k):\QX(\X{P},Z)\equiv\EH(P,Z)
\end{equation}
for every $Z$ in \QX.
\end{theorem}

We can now mention the explicit connection between the two free
constructions we have considered. For that it is useful to prove the
following two lemmata.

\begin{lemma}
Let $P:\cat{C}\op\longrightarrow\Cat{InfSL}$ be an elementary
doctrine. The arrow $(K,k):P\to\X{P}$ preserves quotients,
in the sense that if $q:A\to C$ is a quotient of the $P$-equivalence
relation $\rho$ in $P(A\times A)$, then $K(q):KA\to KC$ is a quotient
of $K_{<K(\pr_1),K(\pr_2)>}(k_{A\times A}(\rho))$.
Therefore, if $P$ has descent quotients of $P$-equivalence relations,
then \X{P} has descent quotients of $\X{P}$-equivalence relations.
\end{lemma}

\begin{proof}
Since $K$ is a quotient functor, it preserves quotients of
$P$-equivalence relations.
Since the $k$-components of $(K,k):P\to\X{P}$ are identity
functions, a \X{P}-equivalence relation $\tau$ on $A$ 
is also a $P$-equivalence relation on $A$.\end{proof}

\begin{lemma}
Let $P:\cat{C}\op\longrightarrow\Cat{InfSL}$ be an elementary
doctrine. If $P$ has comprehensions, then \X{P} has
comprehensions. Moreover $(K,k):P\to\X{P}$ preserves comprehensions,
in the sense that if $\cmp\alpha:X\to A$ is a comprehension of
$\alpha$ in $P(A)$, then $K(\cmp\alpha):KX\to KA$ is a comprehension
of $k_A(\alpha)$.
\end{lemma}

\begin{proof}
Since $P=\X{P}K\op$ and $k$ has identity components, $(K,k)$ preserves
comprehensions. The rest follows immediately.\end{proof}

The results of this section, together with \ref{mthm}, produce an
extension of the quotient completion of 
\cite{MaiettiME:quofcm}.

\begin{theorem}\label{fthm}
There is a left bi-adjoint to the forgetful 2-functor from the full
2-category of \QH on elementary doctrines with comprehensions,
descent quotients and comprehensive diagonals into the 2-category \EH
of elementary doctrines.  
\end{theorem}

\begin{proof}
The left bi-adjoint sends an elementary doctrine
$P:\cat{C}\op\longrightarrow\Cat{InfSL}$ to the elementary
quotient completion
$\X{\Q{\P{P}}}:\cat{X}_{\Q{\P{P}}}\op\longrightarrow\Cat{InfSL}$.\end{proof}

\begin{cor}
For $P:\cat{C}\op\longrightarrow\Cat{InfSL}$ an elementary
doctrine, the elementary quotient completion
$\overline{P}:\cat{Q}_P\op\longrightarrow\Cat{InfSL}$ in
{\rm\cite{MaiettiME:quofcm}} coincides with the doctrine
$\X{\Q{P}}:\cat{X}_{\Q{P}}\op\longrightarrow\Cat{InfSL}$.
\end{cor}

\begin{remark}
Because of the logical setup in \cite{MaiettiME:quofcm}, only a
particular case of \ref{fthm} was proved, namely the left bi-adjoint
was restricted to the full sub-2-category of \EH of elementary
doctrines with full comprehensions and comprehensive diagonals, see
\ref{remce}.
On those doctrines $P:\cat{C}\op\longrightarrow\Cat{InfSL}$, the
action of the left bi-adjoint was simply
$\X{\Q{P}}:\cat{X}_{\Q{P}}\op\longrightarrow\Cat{InfSL}$.
\end{remark}

\section{Comparing some free contructions}

The elementary quotient completion resembles very closely that of
exact completion. In fact, one has the following results.

\begin{theorem}\label{ethm}
Given a cartesian category \cat{S} with weak pullbacks, let
$\Psi:\cat{S}\op\longrightarrow\Cat{InfSL}$ be the
elementary doctrine of 
weak subobjects. Then the doctrine
$\X{\Q{\Psi}}:\cat{X}_{\Q{\Psi}}\op\longrightarrow\Cat{InfSL}$,
is equivalent to the doctrine
$S:\cat{S}\exl\op\longrightarrow\Cat{InfSL}$ of subobjects on the
exact completion 
$\cat{S}\exl$ of \cat{S}.
\end{theorem}

\begin{proof}
It follows from \ref{cov} and the characterization of the embedding of
\cat{S} into $\cat{S}\exl$ in \cite{CarboniA:regec}.
\end{proof}

Though an elementary quotient completion with full comprehension is
regular, see \cite{MaiettiME:quofcm}, the regular completion is an
instance of a completion of a 
doctrine which is radically different from the elementary quotient
completion in \ref{fthm}.

\begin{remark}\label{rthm}
For an elementary dotrine $P:\cat{C}\op\longrightarrow\Cat{InfSL}$,
a \dfn{weak comprehension of $\alpha$} is an arrow
$\cmp\alpha:X\to A$ in \cat{C} such that 
$\tt_{X}\leq P_{\cmp\alpha}(\alpha)$ and, 
for every arrow $g:Y\to A$ such that $\tt_{Y}\leq P_g(\alpha)$ there is
a (not necessarily unique) $h:Y\to X$ such that 
$g=\cmp\alpha\circ h$, see \cite{MaiettiME:quofcm}.

For an elementary dotrine $P:\cat{C}\op\longrightarrow\Cat{InfSL}$
with weak comprehensions, it is possible
to add (strong) comprehensions to its extensional collapse as formal
retracts of weak comprehensions: 
consider the category $\cat{D}_P$ determined by the following data
\begin{description}
\item[objects of $\cat{D}_P$]
are triples $(A,\alpha,c)$ such that $A$ is an object in \cat{C},
$\alpha$ is an object in $P(A)$, and $c:X\to A$ is a weak
comprehension $\alpha$
\item[an arrow] $\ec{f}:(A,\alpha,c)\to (B,\beta,d)$ is an equivalence
class of arrows $f:X\to Y$ in \cat{C} such that
$P_{c\times c}(\delta_A)\leq P_{f\times f}(P_{d\times d}(\delta_B))$
with respect to the relation $f\sim f'$ determined by
$P_{c\times c}(\delta_A)\leq P_{f\times f'}(P_{d\times d}(\delta_B))$
\item[composition] of $\ec{f}:(A,\alpha,c)\to (B,\beta,d)$ and
$\ec{g}:(B,\beta,d)\to (C,\gamma,e)$ is \ec{g\circ f}.
\end{description}
There is a full functor $K:\cat{C}\to\cat{D}_P$ defined on objects
$A$ as $K(A)\colon=(A,\tt_A,\id{A})$---it factors through
$\cat{X}_P$. It 
preserves products and there is an extension
$\R{P}:\cat{D}_P\op\longrightarrow\Cat{InfSL}$ of
$P:C\op\longrightarrow\Cat{InfSL}$ defined on objects as
$\R{P}(A,\alpha,c)\colon=\des{(P_{c\times c}(\delta_A))}$.
The doctrine $\R{P}:\cat{D}_P\op\longrightarrow\Cat{InfSL}$ is
elementary with comprehensions and $K$ preserves all existing
comprehensions. 

Given a cartesian category \cat{S} with weak pullbacks, let
$\Psi:\cat{S}\op\longrightarrow\Cat{InfSL}$ be the
elementary doctrine of 
weak subobjects. Then the doctrine
$\R{\Psi}:\cat{D}_{\Psi}\op\longrightarrow\Cat{InfSL}$
is equivalent to the doctrine
$S:\cat{S}\reg\op\longrightarrow\Cat{InfSL}$ of subobjects on the
regular completion 
$\cat{S}\reg$ of \cat{S}.

The proof is similar to that of \ref{ethm} since,
in the regular completion $\cat{S}\reg$ of \cat{S}, every
object is covered by a regular projective and a subobject of a regular
projective.
\end{remark}

Since the construction given in \ref{ethm} factors through that in
\ref{rthm} via the exact completion of a regular category, see
\cite{FreydP:cata}, and the exact completion of a weakly lex category may
appear very similar to the category $\cat{X}_{\X{\Q{P}}}$, it is 
appropriate to mention an example of an elementary quotient completion
which is not exact.

For that,
consider the indexed poset on the monoid of partial recursive
functions $F:\cat{N}\op\longrightarrow\Cat{InfSL}$ whose value on the
single object of \cat{N} is the powerset of the natural numbers and,
for any $\varphi$ partial recursive function,
$F_\varphi\colon=\varphi^{-1}$, the inverse image of a subset along
the partial map. It is clearly an elementary doctrine, and the
doctrine $\X{\P{F}}:\cat{X}_{\P{F}}\op\longrightarrow\Cat{InfSL}$ is
equivalent to the subobject doctrine 
$S:\cat{PR}\op\longrightarrow\Cat{InfSL}$
on the category \cat{PR} of subsets of
natural numbers and (restrictions of) partial recursive functions
between them, see \cite{CarboniA:somfcr} for properties of that
category, in particular its exact completion (as a weakly lex
category) is the category \cat{D}
of discrete objects of the effective topos.

Now, if one considers the elementary doctrine
$\X{\Q{S}}:\cat{X}_{\Q{S}}\op\longrightarrow\Cat{InfSL}$, the category
$\cat{X}_{\Q{S}}$ is equivalent to the category \cat{PER} of partial
equivalence relations on the natural numbers, and the
indexed poset $\X{\Q{S}}$ is equivalent to that of subobjects
on that category. The category \cat{PER} is not exact because there
are equivalence relations which are not equalizers. In fact, the exact
completion $\cat{PER}\exr$ of \cat{PER} as a regular category is the
category \cat{D} of discrete objects.

Similar examples can be produced using topological categories such
as those in the papers \cite{RosoliniG:typtec,RosoliniG:loccce}. 
Other examples of elementary quotient completions that are not exact
are given in the paper \cite{MaiettiME:quofcm}: one is applied to the 
doctrine of the Calculus of Constructions
\cite{tc90,StreicherT:indipa} and 
the other to the doctrine of the intensional level of the
minimalist foundation in \cite{m09}. 

\bibliographystyle{chicago}
\bibliography{procs,biblio,RosoliniG}
\end{document}